\documentclass[11 pt]{amsart}
\usepackage{amscd,amsfonts,amssymb,amsmath}
\usepackage{hyperref}
\usepackage{epsfig}
\usepackage{mathtools}
\usepackage{mathrsfs}
\usepackage{tikz-cd}

\usepackage{graphicx}

\newtheorem{theorem}{Theorem}[section]
\newtheorem{corollary}[theorem]{Corollary}
\newtheorem{lemma}[theorem]{Lemma}
\newtheorem{proposition}[theorem]{Proposition}
\theoremstyle{definition}

\newtheorem{remark}[theorem]{Remark}
\numberwithin{equation}{subsection}

\newcommand{\Aut}{\operatorname{Aut}}

\newcommand{\Ker}{\operatorname{Ker}}

\newcommand{\Inn}{\operatorname{Inn}}

\newcommand{\R}{\operatorname{R}}
\newcommand{\M}{\operatorname{M}}

\newcommand{\Z}{\operatorname{Z}}

\newcommand{\I}{\operatorname{I}}

\begin{document}

\title[Some Remarks on Twin Groups]{Some Remarks on Twin Groups}

\author{Tushar Kanta Naik}
\address{Department of Mathematical Sciences, Indian Institute of Science Education and Research (IISER) Mohali, Sector 81,  S. A. S. Nagar, P. O. Manauli, Punjab 140306, India.}
\email{mathematics67@gmail.com / tushar@iisermohali.ac.in}
\author{Neha Nanda}
\email{nehananda@iisermohali.ac.in}
\author{Mahender Singh}
\email{mahender@iisermohali.ac.in}


\begin{abstract}
The twin group $T_n$ is a right angled Coxeter group generated by $n- 1$ involutions and having only far commutativity relations. These groups can be thought of as planar analogues of Artin braid groups. In this note, we study some properties of twin groups whose analogues are well-known for Artin braid groups. We give an algorithm for two twins to be equivalent under individual Markov moves. Further, we show that twin groups $T_n$ have $R_\infty$-property and are not co-Hopfian for $n \ge 3$.
\end{abstract}
\maketitle

\section{Introduction}\label{introduction}
The twin group $T_n$, $n \ge 2$, is a right angled Coxeter group with $n-1$ generators and only far commutativity relations. These groups first appeared in the work of Shabat and Voevodsky \cite{sv} in the context of curves over number fields. Later, these groups were investigated by Khovanov \cite{Khovanov}, who referred them as twin groups and gave a geometric interpretation of these groups similar to the one for Artin braid groups $B_n$. Khovanov considered configurations of $n$ arcs in the infinite strip $\mathbb{R} \times  [0,1]$ connecting $n$ marked points on each of the parallel lines $\mathbb{R} \times \{1\}$ and $\mathbb{R} \times \{0\}$ such that each arc is monotonic and no three arcs intersect at a common point. Two such configurations are considered equivalent if one can be deformed into the other by a homotopy of such configurations in $\mathbb{R} \times [0,1]$ keeping the end points of the arcs fixed, and an equivalence class is called a \textit{twin}. Analogous to braid groups, the product of two twins is defined by placing one twin on top of the other. The collection of twins on $n$ strands under this operation forms a group which is isomorphic to $T_n$.
\par
Analogous to closure of a geometric braid in the $3$-space, one can define the closure of a twin on a $2$-sphere by taking the one point compactification of the plane. A \textit{doodle} on a closed oriented surface is a collection of piecewise linear closed curves without triple intersections. It follows that closure of a twin on a $2$-sphere is a doodle. Doodles on a 2-sphere were first introduced by Fenn and Taylor \cite{Fenn-Taylor}, and the idea was extended to immersed circles in a 2-sphere by Khovanov \cite{Khovanov}. He proved an analogue of the classical Alexander theorem for doodles. On the other hand, an analogue of the Markov theorem for doodles on a $2$-sphere has been established by Gotin \cite{Gotin}. Recently, Bartholomew-Fenn-Kamada-Kamada \cite{bfkk, fenn} extended the study of doodles to immersed circles in a closed oriented surface of any genus, which can be thought of as virtual links analogue for doodles. A recent preprint \cite{CisnerosFloresJuyumaya2020} constructs an Alexander type invariant for
oriented doodles from a deformation of the classical Tits representation of $T_n$ and prove that analogous to the Alexander polynomial for classical links, this invariant vanishes on unlinked doodles with more than one component.
\par
Recall that, in view of the classical Alexander and Markov theorems, the problem of classifying isotopy classes of links in the $3$-space is equivalent to the algebraic problem of classifying Markov classes of braids in the infinite braid group. The algebraic link problem is to determine whether two braids are equivalent under Markov moves. The conjugacy problem (the first Markov equivalence)  in braid groups was solved by Garside \cite{Garside}, whereas the case of second Markov equivalence was attempted by Humphries \cite{Humphries1}. Since we have Alexander and Markov theorems for doodles on a 2-sphere, it is natural to formulate an analogue of the algebraic link problem for doodles. We address this problem in Theorem \ref{MainTheoremAlgebraicDoodleProblem} of Section  \ref{AlgebraicDoodleProblem}.

\par
We also study $R_{\infty }$-property and (co)-Hopfianity of twin groups. A group $G$ is said to have $R_{\infty }$-property if it has infinitely many $\phi$-twisted conjugacy classes for each automorphism $\phi$ of $G$. Here, two elements $x, y \in G$ are said to lie in the same $\phi$-twisted conjugacy class if there exists $g \in G$ such that $x = gy\phi(g)^{-1}$. The idea of twisted conjugacy arose from the work of Reidemeister \cite{Reidemeister}. Due to deep connection of twisted conjugacy with Nielsen fixed-point theory, the study of $R_{\infty }$-property of groups has attracted a lot of attention. It is well-known that braid groups $B_n$ have $R_{\infty }$-property for all $n \geq 3$ \cite{Felshtyn}. We refer the reader to \cite{Cox, Timur3, Timur2, Goncalves2, Goncalves1,  Juhasz, Mubeena, Timur4, Timur1} for recent works on the topic. In Section \ref{Rinfinity}, we show that twin groups $T_n$ have $R_{\infty }$-property for $n \geq3$ (Theorem \ref{MainTheoremRinfinity}).
\par
Recall that a group is co-Hopfian (respectively Hopfian) if every injective (respectively surjective) endomorphism is an automorphism. These properties are closely related to $R_{\infty }$-property, see, for example \cite[Lemma 2.3]{Mubeena2}. Braid groups $B_n$ are known to be Hopfian being residually finite \cite[Chapter I, Corollary 1.22]{Kassel} and are not co-Hopfian for $n \geq 2$ \cite{bell}. In fact, the map which sends each standard generator of $B_n$ to itself times a fixed power of the central element extends to an injective endomorphism which is not surjective. It is well-known that Coxeter groups, in particular twin groups, are Hopfian \cite[Theorem C, p. 55]{brown}. In Section \ref{Hopfian}, we prove that twin groups $T_n$ are co-Hopfian only for $n = 2$ (Theorem \ref{MainTheoremCoHopfian}). We compare our results with braid
groups and mention some open problems for the reader.
\bigskip

\section{Preliminaries}\label{basic}
We begin the section by setting some notations. For elements $g, h$ of a group $G$, we denote $g^{-1}h^{-1}gh$ by $[ g, h ]$ and the conjugacy class of $g$ by $g^G$.
\par
For an integer $n \ge 2$, the \textit{twin group} $T_n$ is defined as the group with the presentation
$$\big\langle s_1, s_2, \dots, s_{n-1}~|~ s_i^{2} = 1~ \text{for}~1 \le i \le n-1~\textrm{and}~ s_is_j = s_js_i~ \text{for}~ |i-j| \geq 2 \big\rangle.$$
The generator $s_i$ can be geometrically presented by a configuration of $n$ arcs as shown in Figure \ref{figure1}.
\begin{figure}[hbtp]
\centering
\includegraphics[scale=0.5]{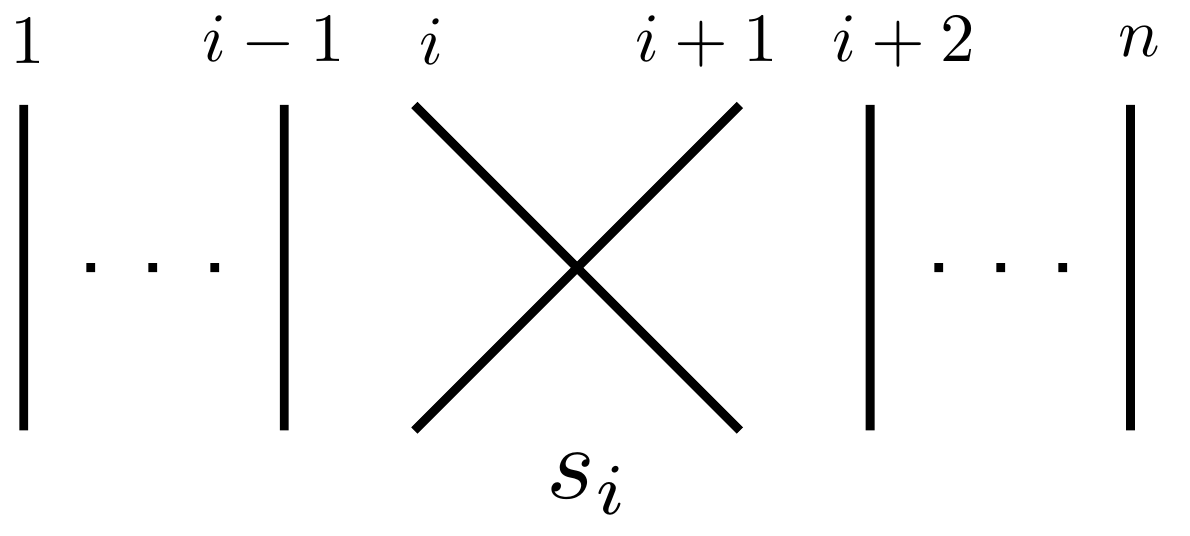}
\caption{The generator $s_i$}
\label{figure1}
\end{figure}

The kernel of the natural surjection from $T_n$ onto $S_n$, the symmetric group on $n$ symbols, is called the \textit{pure twin group} and is denoted by $PT_n$.
\par

\subsection{Elementary transformations and cyclically reduced words} We recall some basic ideas and results from \cite{NNS} on twin groups that will be used in subsequent sections. We begin by defining three elementary transformations of a word $w\in T_n$ as follows:
\begin{itemize}
\item[(i)] Deletion: Replace the word $w$ by deleting subword of the form $s_is_i$.
\item[(ii)] Insertion: Replace the word $w$ by inserting a word of the form $s_is_i$ in $w$.
\item[(iii)]Flip: Replace a subword of the form $s_is_j$ by $s_js_i$ in $w$ whenever $\mid i-j\mid \geq 2$.
\end{itemize}

We say that two words $w_1$ and $w_2$ are {\it equivalent}, written as $w_1 \sim w_2$, if there is a finite chain of elementary transformations turning $w_1$ into $w_2$. It is easy to check that $\sim$ is an equivalence relation. Obviously, $w\sim w'$ if and only if both $w$ and $w'$ represent the same element in $T_n$.\\
\par

For a given word $w=s_{i_1} s_{i_2}\dots s_{i_k}$ in $T_n$, let $\ell(w)=k$ be the \textit{length} of $w$. We say that the word $w$ is {\it reduced} if $\ell(w)\leq \ell(w')$ for all $w'\sim w$. By well-ordering principle,  the equivalence class of each word contains a reduced word. It is possible to have more than one reduced word representing the same element. In fact, all reduced words representing the same element differ by finitely many flip transformations, for example, $s_1s_4$ and $s_4s_1$. Obviously, any two reduced words in the same equivalence class have the same length. This allows us to define  the length of an element $ w \in T_n$ as the length of a reduced word representing $w$. Below is a characterisation of a reduced word.

\begin{lemma}\label{lem1}
A word $w\in T_n$ is reduced  if and only if $w$ satisfies the property that whenever two $s_{i}$'s appear in $w$, there always exist atleast one $s_{i-1}$ or $s_{i+1}$ in between them.
\end{lemma}

A \textit{cyclic permutation} of a word  $w=s_{i_1} s_{i_2}\dots s_{i_k}$ (not necessarily reduced) is a word $w'$ (not necessarily distinct from $w$) of the form $s_{i_t}s_{i_{t+1}} s_{i_{t+2}}\dots s_{i_k}s_{i_1} s_{i_2}\cdots s_{i_{t-1}}$ for some $1\leq t \leq k$. It is easy to see that $w'=(s_{i_1} s_{i_2}\dots s_{i_{t-1}})^{-1}w(s_{i_1} s_{i_2}\dots s_{i_{t-1}})$, that is, $w$ and $w'$ are conjugates of each other. A word $w$ is called \textit{cyclically reduced} if each cyclic permutation of $w$ is reduced. It is immediate that a cyclically reduced word is reduced,  but the converse is not true. For example, $s_1s_2s_1$ is reduced but not cyclically reduced.

\begin{corollary}\label{cor1}
Each word in $T_n$ is conjugate to some cyclically reduced word.
\end{corollary}

\begin{theorem}\label{conditionforconjugate}
 Let $w_1, w_2$ be two cyclically reduced words in $T_n$. Then, $w_1$ is conjugate to $w_2$ if and only if they are cyclic permutation of each other modulo finitely many flip transformations.
\end{theorem}
\bigskip
\subsection{Doodles on a $2$-sphere}
A \textit{doodle} on a $2$-sphere with $m$ components is a collection of $m$ piecewise linear closed curves without triple or higher intersection points. Two doodles are said to be \textit{equivalent} if one can be obtained from the other by an isotopy of the $2$-sphere and a finite sequence of local moves $R_1$ and $R_2$ as shown in the Figure \ref{figure2}. 
\begin{figure}[hbtp]
\centering
\includegraphics[scale=.3]{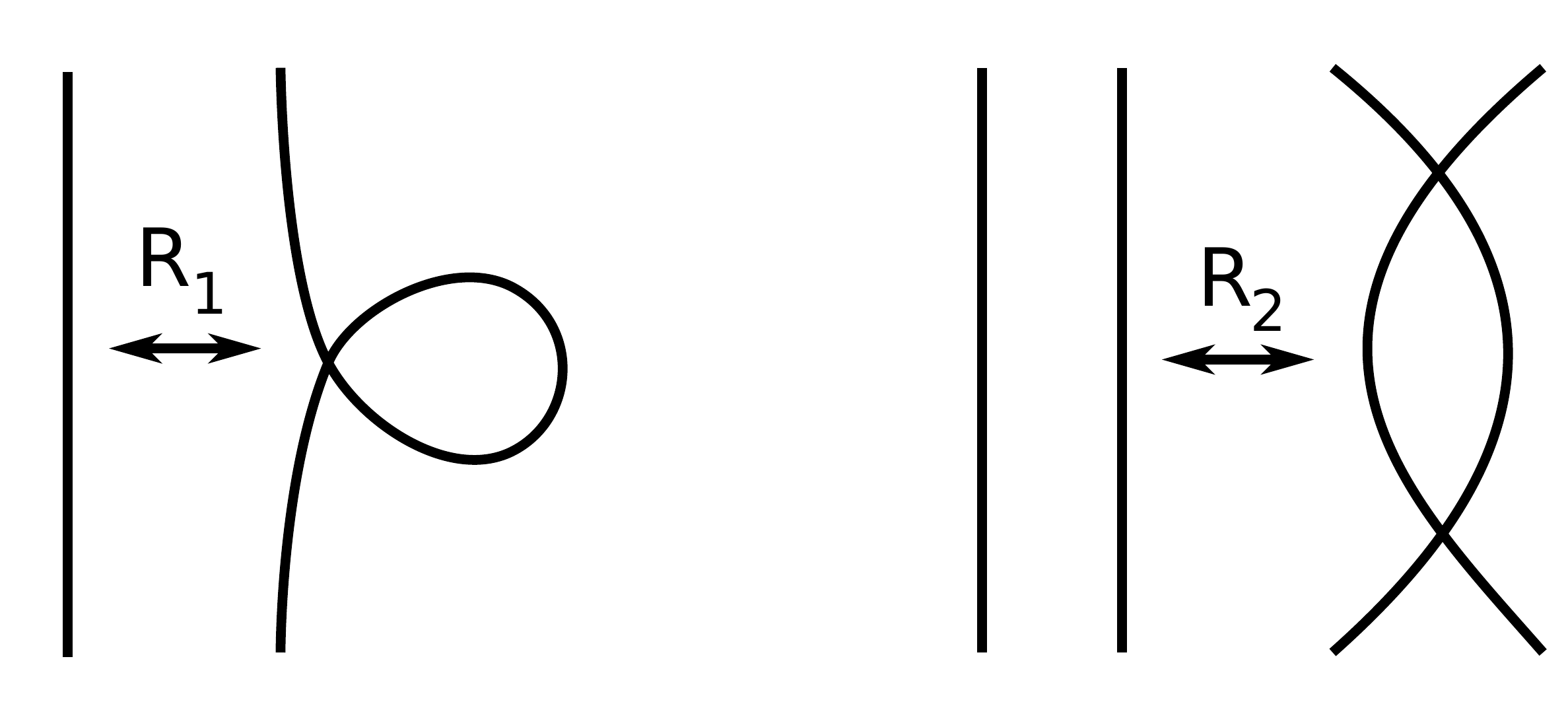}
\caption{The $R_1$ and $R_2$ moves.}
\label{figure2}
\end{figure}
An \textit{oriented doodle} is a doodle with orientation on each of its components. By the \textit{closure} of a twin $\beta$, we mean a diagram obtained by joining the end points of $\beta$ on a $2$-sphere (by taking the one point compactification of the plane) as illustrated in Figure \ref{figure3}.\\

\begin{figure}[hbtp]
\centering
\includegraphics[scale=1]{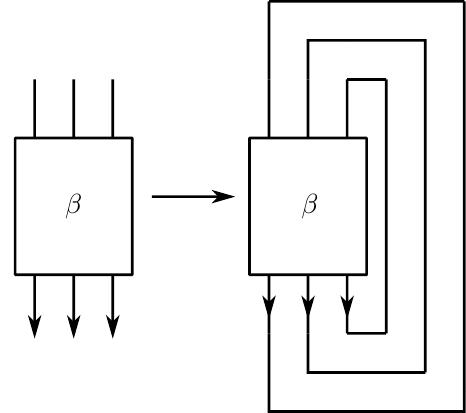}
\caption{The closure of a twin $\beta$.}
\label{figure3}
\end{figure}

It is easy to check that the closure of a twin is a doodle. Further, an orientation on a twin gives an orientation on its closure. The analogue of the classical Alexander theorem in this setting has been established in \cite[Theorem 2.1]{Khovanov}.

\begin{theorem}
Every oriented doodle on a $2$-sphere is the closure of some twin.
\end{theorem}

For twins $\alpha$ and $\beta$ (possibly with different number of strands), we denote by $\alpha \otimes \beta$, the twin obtained by adding the diagram of $\alpha$ to the left of the diagram of $\beta$. For any positive integer $n$ and $\beta \in T_n$, define the following moves:
\begin{itemize}
\item[$\M_1$] : $\beta \otimes \I$ $\to$ $\I \otimes ~\beta$,

\item[$\M_2$] :  $\beta$ $\to$ $\alpha^{-1} \beta \alpha$, 
\item[$\M_3$] :  $\beta$ $\to$ $(\beta \otimes I)s_ns_{n-1} \dots s_{i+1}s_{i}s_{i+1} \dots s_{n-1}s_{n}$,
\item[$\M_4$] :  $\beta$ $\to$ $(I \otimes \beta)s_1s_2 \dots s_{i-1}s_{i}s_{i-1} \dots s_2s_1$,
\end{itemize}
where $\I$ is a twin with only one strand, $\alpha \in T_n$ and $s_i \in T_{n+1}$.\\
Two twins are said to be \textit{$M$-equivalent} if one can be obtained from the other by a finite sequence of moves $\M_1 - \M_4$ and their inverses.

\begin{figure}[hbtp]
\centering
\includegraphics[scale=.2]{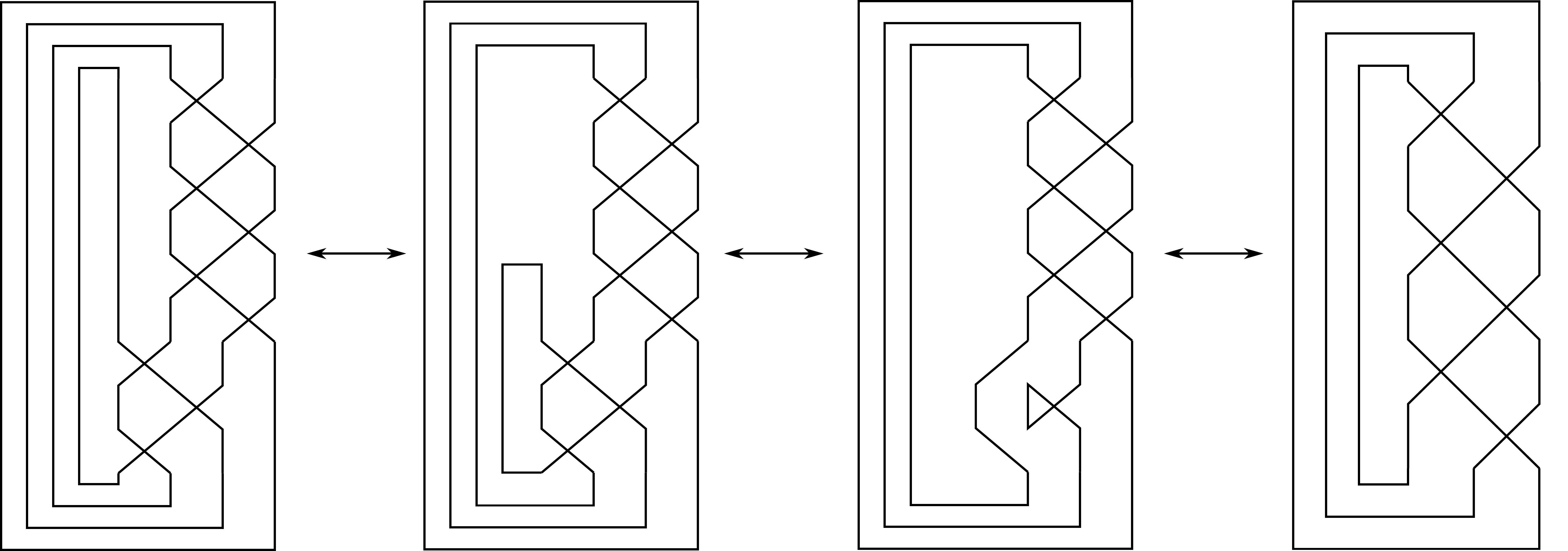}
\caption{The closures of $(s_1s_2)^3$ and $(s_2s_3)^3s_1s_2s_1$ being equivalent as doodles.}
\label{figure4}
\end{figure}

If $\M_i(\beta)$ is the twin obtained from $\beta$ by applying the $\M_i$-move, then it is easy to check that the closure of $\M_i(\beta)$ is equivalent to the closure of $\beta$. For example, the closure of $(s_1s_2)^3 \in T_3$ and the closure of $(s_2s_3)^3s_1s_2s_1 \in T_4$ are equivalent by $\M_4$-move as shown in Figure \ref{figure4}. The converse has been established by Gotin  \cite[Theorem 4.1]{Gotin}.
\begin{theorem}
Any two twins with equivalent closures are $M$-equivalent.
\end{theorem}
\medskip

\section{Algebraic doodle problem}\label{AlgebraicDoodleProblem}
In case of classical links and braids, the algebraic link problem asks whether given two braids are
equivalent under the classical Markov moves. The algebraic doodle problem can be formulated
along the similar lines, that is, to determine whether two twins are equivalent under the Markov
moves $\M_1-\M_4$.
\par 
It is easy to note that the $\M_1$-move is equivalent to saying that whenever the reduced expression of $\alpha = s_{i_1}s_{i_2}\dots s_{i_k} \in T_{n+1}$ does not contain $s_n$, we can replace $\alpha$ by $\I \otimes \alpha = s_{{i_1}+1}s_{{i_2}+1}\dots s_{{i_k}+1}$. Next, checking the equivalence of twins under the $\M_2$-move is same as the conjugacy problem which is solvable in $T_n$ \cite{Krammer, NNS}. Thus, we consider the moves $\M_3$ and $\M_4$ and prove the following result.

\begin{theorem}\label{MainTheoremAlgebraicDoodleProblem}
Given a twin $\alpha \in T_{n+1}$, there is an algorithm to determine whether
\begin{itemize}
\item[(1)] $\alpha$ can be written as $(\beta \otimes \I)s_ns_{n-1} \dots s_{i+1}s_{i}s_{i+1} \dots s_{n-1}s_{n}$,
\item[(2)] $\alpha$ can be written as $(I \otimes \beta)s_1s_2 \dots s_{i-1}s_{i}s_{i-1} \dots s_2s_1$,
\end{itemize}
for some $\beta \in T_n$ and $1\le i\le n$.
\end{theorem}
\begin{proof}

{{Case (1).}} We determine whether $\alpha \in T_{n+1}$ can be written as $(\beta \otimes \I)s_ns_{n-1} \dots s_{i+1}s_{i}s_{i+1} \dots s_{n-1}s_{n}$, for some $\beta \in T_n$ and $1 \le i \le n$. Upon applying Lemma \ref{lem1}, we get a reduced word equivalent to $\alpha$ and have the following possibilities:
\begin{itemize}
\item[(i)] If there is only one $s_n$ present in the reduced expression of $\alpha$, then we can write $\alpha$ as $\alpha's_n\alpha''$, where $\alpha', \alpha'' \in T_n$. Such an $\alpha$ can be written in the desired form if and only if  there is no $s_{n-1}$ present in $\alpha''$.

\item[(ii)] Suppose that there are two $s_n$'s present in the reduced expression of $\alpha$. If the expression does not have a subword of the form $s_ns_{n-1} \dots s_{i+1}s_{i}s_{i+1} \dots s_{n-1}s_{n}$ for any $1 \le i \le n-1$, then we cannot write $\alpha$ in the desired form. On the other hand, if the reduced expression of $\alpha$ can be written as $\alpha's_ns_{n-1} \dots s_{i+1}s_{i}s_{i+1} \dots s_{n-1}s_{n}\alpha''$, then $\alpha$ has the desired form if and only if $\alpha''$ is a word in $s_j$ for $1 \le j \le i-2$.

\item[(iii)] If the number of $s_n$'s present in the expression is greater than equal to $3$, then  we cannot write $\alpha$ in the desired form. For, if we get a subword of the form $s_ns_{n-1} \dots s_{i+1}s_{i}s_{i+1} \dots s_{n-1}s_{n}$ for some $i$ and we move this subword to the rightmost position in the reduced expression of $\alpha$ by flip transformations, there will be an $s_n$ present in the expression of $\beta$ which is not possible since $\beta \in T_n$.\\
\end{itemize}

{{Case (2).}} It is easy to note that if $\beta \in T_n$ is written as $s_{i_1} s_{i_2}\dots s_{i_k}$, then $I \otimes \beta = s_{{i_1}+1} s_{{i_2}+1} \dots s_{{i_k}+1} \in T_{n+1}$. We determine whether $\alpha \in T_{n+1}$ can be written as $\alpha = (I \otimes \beta)s_1s_2 \dots s_{i-1}s_{i}s_{i-1} \dots s_2s_1$, for some $\beta \in T_n$ and $1\le i\le n$. On applying Lemma \ref{lem1}, we get a reduced word equivalent to $\alpha$ and  have the following possibilities:
\begin{itemize}

\item[(i)] If there is only one $s_1$ present in the reduced expression of $\alpha$, then we can write $\alpha$ as $\alpha's_1\alpha''$, where $\alpha', \alpha'' \in T_{n+1}$. Such an $\alpha$ can be written in the desired form if and only if  there is no $s_{2}$ present in the expression of $\alpha''$.

\item[(ii)] Suppose that there are two $s_1$'s present in the reduced expression of $\alpha$. If the expression does not have a subword of the form $s_1s_2 \dots s_{i-1}s_{i}s_{i-1} \dots s_2s_1$ for any $2 \le i \le n$, then we cannot write $\alpha$ in the desired form. On the other hand, if we can write reduced expression of $\alpha$ as $\alpha's_1s_2 \dots s_{i-1}s_{i}s_{i-1} \dots s_2s_1\alpha''$, then $\alpha$ has the desired form if and only if $\alpha''$ is a word in $s_j$ for $i+2 \le j \le n$.

\item[(iii)] If number of $s_1$'s present in the expression is greater than equal to $3$, then we cannot write $\alpha$ in the desired form. For, if we get a subword of the form $s_1s_2 \dots s_{i-1}s_{i}s_{i-1} \dots s_2s_1$ for some $i$, there will be an $s_1$ present in the expression of $\I \otimes \beta$ which is not possible, since it is an expression in $s_2, s_3, \dots , s_n$.
\end{itemize}
\end{proof}

We now define split doodles and split twins analogous to split links and braids. A doodle on a $2$-sphere is said to be \textit{split} if it contains two disjoint open disks each containing at least one component of the doodle. We define a twin to be \textit{split} if its closure is a split doodle on a $2$-sphere. The following figure is an example of a split doodle which is the closure of a twin $(s_1s_2)^3(s_4s_5)^4$.
\begin{figure}[hbtp]\label{figure}
\centering
\includegraphics[scale=3.5]{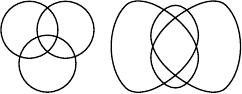}
\caption{A split doodle}
\end{figure}

For each $1 \le i\le n-1$, let $T_n^i$ be the subgroup of $T_n$ generated by $\{s_1, s_2, \dots, s_{i-1}, s_{i+1}, \dots, s_{n-1}\}$. The following proposition, whose proof is immediate, gives sufficient conditions for a twin to be split.
\begin{proposition}
If $\alpha \in T_n$ satisfy one of the following conditions:
\begin{enumerate}
\item $\alpha$ is conjugate to a word in $T_n^i$ for some $1 \le i \le n-1$,
\item $\alpha = (\beta \otimes \I)s_{n-1}s_{n-2} \dots s_{j+1}s_{j}s_{j+1} \dots s_{n-2}s_{n-1}$, where $\beta$ is conjugate of a word in $T_{n-1}^i$, $1 \le  i \le n-2$ and  $1 \le j \le n-1$,
\item $\alpha = (\I \otimes \beta)s_1s_2 \dots s_{j-1}s_{j}s_{j-1} \dots s_2s_1$, where $\beta$ is conjugate of a word in $T_{n-1}^i$, $1 \le i \le n-2$ and  $1 \le j \le n-1$,
\end{enumerate}
then $\alpha$ is a split twin.
\end{proposition}

\section{$R_{\infty }$-property for twin groups}\label{Rinfinity}
Let $G$ be a group and $\phi$ an automorphism of $G$. Two elements $x, y\in G$ are said to be ($\phi$-twisted) $\phi$-conjugate if there exists an element $g\in G$ such that $x = gy\phi(g)^{-1}$. The relation of $\phi$-conjugation is an equivalence relation and it divides the group into $\phi$-conjugacy classes. Taking $\phi$ to be the identity automorphism gives the usual conjugacy classes. The number of $\phi$-conjugacy classes $R(\phi)\in \mathbb{N} \cup \{\infty\}$ is called the \textit{Reidemeister number} of the automorphism $\phi$. We say that a group $G$ has $R_{\infty }$-property if $\R(\phi)= \infty$ for each $\phi \in \Aut(G)$. Obviously, finite groups (in particular $T_2$) do not satisfy $R_{\infty}$-property. In this section, we prove that twin groups $T_n$ have $R_{\infty }$-property for each $n \geq 3$. We begin by recalling a basic result on twisted conjugacy classes \cite[Corollary 3.2]{Fel'shtyn-Leonov-Troitsky}.

\begin{lemma}\label{sufficient-condition}
Let $\phi$ be an automorphism and $\hat{g}$ an inner automorphism of a group $G$. Then $R(\hat{g} \phi) = R(\phi)$.
\end{lemma}

The following result relates twisted conjugacy with usual conjugacy.

\begin{lemma}\label{alternate-condition}
Let $G$ be a group and $\phi$ an order $k$ automorphism of $G$. If $x, y\in G$ are $\phi$-conjugates, then the elements $x \phi(x) \phi^2(x)\cdots \phi^{k-1}(x)$ and  $y \phi(y) \phi^2(y)\cdots \phi^{k-1}(y)$ are conjugates (in the usual sense). The converse is not true in general.
\end{lemma}

\begin{proof}
Since $x, y\in G$ are $\phi$-conjugates, there exists $z\in G$ such that $x=zy\phi (z^{-1})$. Applying $\phi^i$, $1 \le i \le k-1$, to this equality gives	
\begin{align*}
\phi(x) & = \phi(z) \phi(y)\phi^2 (z^{-1}),&\\
\phi^2(x) & = \phi^2(z) \phi^2(y)\phi^3 (z^{-1}),&\\
\vdots \\
\phi^{k-1}(x) & = \phi^{k-1}(z)  \phi^{k-1}(y)\phi^k (z^{-1}) =  \phi^{k-1}(z)  \phi^{k-1}(y) z^{-1}.&
\end{align*}
Multiplying the preceding equalities gives $x \phi(x) \phi^2(x)\cdots \phi^{k-1}(x) = z \big ( y \phi(y) \phi^2(y)\cdots \phi^{k-1}(y)\big ) z^{-1}$, which is the first assertion.
\par

For the second assertion, consider the extra-special $p$-group $$\mathscr{G} = \big\langle a, b, c ~|~ a^p=b^p=c^p=1, ab=bac, ac=ca, bc=cb \big\rangle$$ of order $p^3$ and exponent $p$, where $p$ is an odd prime. Note that $\mathscr{G}' = \Z(\mathscr{G}) =\langle c \rangle$ is of order $p$.  It is easy to check that the map $\phi : \mathscr{G}\rightarrow \mathscr{G}$ given by $\phi(a) = ac$ and $\phi(b)=bc$ extends to an order $p$ automorphism of $\mathscr{G}$. Then, we have $a \phi(a)\phi^2(a)\cdots \phi^{p-1}(a) = 1 =  b \phi(b)\phi^2(b)\cdots \phi^{p-1}(b)$. Suppose that there exists $g\in \mathscr{G}$ such that $a = g b \phi (g^{-1})$. This gives $a = g b g^{-1}c^l$ for some $l\in \mathbb{Z}$. Thus, $a = b b^{-1}g b g^{-1}c^l = b[b, g^{-1}]c^l \in b\mathscr{G}'$, which is not possible. Hence, $a$ cannot be $\phi$-conjugate to $b$.
\end{proof}

The following result \cite[Theorem 6.1]{NNS} will be used in the proof of the main theorem.

\begin{theorem}\label{T_n}
Let $T_n$ be the twin group with $n\geq 3$. Then the following hold:
\begin{enumerate}
\item  $\Aut(T_3) = \Inn(T_3) \rtimes \langle \psi \rangle \cong \Inn(T_3) \rtimes \mathbb{Z}_2$,
\item  $\Aut(T_4) = \Inn(T_4) \rtimes \langle\psi, \tau \rangle \cong \Inn(T_4) \rtimes S_3$,
\item $\Aut(T_n) = \Inn(T_n) \rtimes \langle \psi, \kappa \rangle \cong \Inn(T_n) \rtimes D_8$ for $n \geq 5$,
\end{enumerate}
where
\begin{enumerate}
\item[]  $D_8$ is the dihedral group of order 8,
\item[]  $\psi:T_n \to T_n$ is given by $\psi(s_i)=s_{n-i}$ for each $1\leq i\leq n-1$,
\item[] $\tau : T_4 \to T_4$ is given by $\tau(s_1)=s_1s_3$, $\tau(s_2)=s_2$ and $\tau(s_3)=s_1$,
\item[] $\kappa : T_n \to T_n$ is given by $\kappa(s_3)=s_{n-3}s_{n-1}$ and $\kappa(s_i)=s_{n-i}$ for $i\neq 3$.
\end{enumerate}
\end{theorem}

We now proceed to prove the main theorem of this section.

\begin{theorem}\label{MainTheoremRinfinity}
$T_n$ satisfy $R_{\infty}$-property for all $n\geq 3$.
\end{theorem}

\begin{proof}
It follows from \cite[Proposition 5.1]{NNS} that $T_n$ has infinitely many conjugacy classes for all $n\geq 3$. Hence, due to Lemma \ref{sufficient-condition}, it is enough to show that $T_n$ has infinitely many $\phi$-conjugacy classes for automorphisms $\phi$ in the groups $ \langle \psi \rangle$, $\langle\psi, \tau \rangle$ and $\langle \psi, \kappa \rangle$, which are all finite. Our plan is to use Lemma \ref{alternate-condition}.  We divide the proof into four cases, namely,  $n\geq 6$, $n=5$, $n=4$ and $n=3$.\\

\noindent {{Case $n \ge 6.$}} Consider the sequence of elements $x_i = (s_1s_2)^i$, $i\geq 1$. We claim that for any automorphism $\phi \in \langle \psi, \kappa \rangle$, $x_i$ is not $\phi$-conjugate to $x_j$ whenever $i\neq j$. Let us, on the contrary, suppose that $x_i$ is $\phi$-conjugate to $x_j$ for some $i\neq j$. Then, by Lemma \ref{alternate-condition},  $x_i \phi(x_i) \phi^2(x_i)\cdots \phi^{k-1}(x_i)$ and  $x_j \phi(x_j) \phi^2(x_j)\cdots \phi^{k-1}(x_j)$ are conjugates, where $k$ is the order of the automorphism $\phi$.
\par
Note that the subgroup $H= \langle s_1, s_2, s_{n-2}, s_{n-1} \rangle$ of $T_n$ is invariant under all automorphisms in $\langle \psi, \kappa \rangle$. In fact, $\psi (x) = \kappa (x)$ for all $x \in H$. Thus, it is sufficient to show that $x_i \psi(x_i)$ and  $x_j \psi(x_j)$ are not conjugates in $T_n$. Observe that $x_i \psi(x_i)= (s_1s_2)^i (s_{n-1}s_{n-2})^i$ and $x_j \psi(x_j)= (s_1s_2)^j (s_{n-1}s_{n-2})^j$. It is easy to see that the words $(s_1s_2)^i (s_{n-1}s_{n-2})^i$ and $(s_1s_2)^j (s_{n-1}s_{n-2})^j$ are cyclically reduced, and hence by Theorem \ref{conditionforconjugate} they are not conjugates of each other for $i\neq j$. Thus, we have infinitely many $\phi$-conjugacy classes in $T_n$, $n\ge 6$, for any automorphism $\phi$ of $T_n$.\\

\noindent {{Case $n= 5.$}}  For this case we consider the sequence of elements $x_i = (s_1s_2)^{2i}$, $i\geq 1$. We claim that for any automorphism $\phi \in \langle \psi, \kappa \rangle$, $x_i$ is not $\phi$-conjugate to $x_j$ for $i\neq j$. Note that there are two automorphisms of order $4$, namely $\kappa$ and $\kappa^3$, and five automorphisms of order $2$, namely $\kappa^2$, $\psi \kappa$, $\psi \kappa^2$, $\psi \kappa^3$ and $\psi$.\\
Direct computations give
$$x_i\kappa(x_i)\kappa^2(x_i)\kappa^3(x_i) = (s_1s_2)^{2i}(s_4s_3)^{2i}(s_1s_2s_4)^{2i}(s_4s_1s_3)^{2i} = (s_1s_2)^{2i}(s_4s_3)^{2i}(s_1s_2)^{2i}(s_4s_3)^{2i}.$$ Again, by Theorem \ref{conditionforconjugate}, $(s_1s_2)^{2i}(s_4s_3)^{2i}(s_1s_2)^{2i}(s_4s_3)^{2i}$ and $(s_1s_2)^{2j}(s_4s_3)^{2j}(s_1s_2)^{2j}(s_4s_3)^{2j}$ are not conjugate for $i\neq j$. Similarly,
$x_i\psi\kappa^2(x_i)=(s_1s_2)^{2i}(s_4s_1s_3)^{2i} = (s_1s_2)^{2i}(s_4s_3)^{2i}$. As before, $x_i\psi\kappa^2(x_i)$ is not conjugate to $x_j\psi\kappa^2(x_j)$ for $i\neq j$. The remaining automorphisms can be considered in the same manner, and hence there are infinitely many $\phi$-conjugacy classes in $T_5$ for any automorphism $\phi$.\\

\noindent {{Case $n= 4.$}}  We again consider the sequence of elements $x_i = (s_1s_2)^i$, $i\geq 1$, and prove that for any automorphism $\phi \in \langle\psi, \tau \rangle$, $x_i$ is not $\phi$-conjugate to $x_j$ whenever $i\neq j$. Note that there are two automorphisms of order $3$, namely $\tau$ and $\tau^2$, and three automorphisms of order $2$, namely $\psi \tau$, $\psi \tau^2$ and $\psi$. Direct computations yield $x_i\tau(x_i)\tau^2(x_i) = (s_1s_2)^i (s_1s_3s_2)^i (s_3s_2)^i$. Again by Theorem \ref{conditionforconjugate}, $x_i\tau(x_i)\tau^2(x_i)$ is not conjugate to $x_j\tau(x_j)\tau^2(x_j)$ whenever $i\neq j$. The remaining automorphisms can be dealt with similarly, and the assertion follows.\\

\noindent {{Case $n= 3.$}} Unlike the earlier cases, here we consider the sequence of elements $x_i = (s_1s_2)^is_1$, $i\geq 1$. In this case we need to consider only one automorphism $\psi$ which is of order $2$. We have $x_i\psi(x_i) = (s_1s_2)^is_1 (s_2s_1)^is_2 = (s_1s_2)^{2i+1}$. By Theorem  \ref{conditionforconjugate}, $x_i\psi(x_i) $ is not conjugate to $x_j\psi(x_j)$ whenever $i\neq j$, and hence there are infinitely many $\psi$-conjugacy classes in $T_3$. This completes the proof of the theorem.
\end{proof}

\begin{remark}\label{remark}
It is interesting to see whether the pure twin group $PT_n$ has $R_{\infty}$-property. It is known that $PT_3 \cong \mathbb{Z}$, $PT_4 \cong F_7$ and $PT_5 \cong F_{31}$ \cite{bardakov,Jesus}. It follows from \cite[Theorem 2.1]{Karel} and \cite[Lemma 2.1]{Mubeena2} that non-abelian free groups of finite rank have $R_{\infty}$-property. Thus, $PT_4$ and $PT_5$ have $R_{\infty}$-property. A precise description of $PT_6$ has been obtained recently in \cite[Theorem 2]{Mostovoy} where it is proved that $PT_6 \cong F_{71} *(*_{20}(\mathbb{Z} \oplus \mathbb{Z}))$. On the other hand, a complete description of $PT_n$ is still unknown for $n \geq 7$. Thus, determining whether $PT_n$ has $R_{\infty}$-property for $n \geq 6$ remains an open problem.
\end{remark}

\begin{remark}
It is well-known that the braid group $B_n$, $n \geq 3$, has $R_{\infty}$-property. It is interesting to explore the property for the pure braid group $P_n$. It is known that $P_2$ is infinite cyclic, and hence does not have $R_{\infty}$-property. It follows from \cite[Section 4]{Mahender} that $P_3 \cong F_2 \times \mathbb{Z}$. Since $\mathbb{Z}$ is characteristic in $P_3$ and $F_2$ has $R_{\infty}$-property (as in Remark \ref{remark}), by \cite[Lemma 2.1]{Mubeena2} it follows that $P_3$ has $R_{\infty}$-property. Whether the pure braid group $P_n$ has $R_{\infty}$-property for $n \geq 4$ seems to be unknown.
\end{remark}

\section{Co-Hopfian property for twin groups}\label{Hopfian}
A group is said to be co-Hopfian (respectively Hopfian) if every injective (respectively surjective) endomorphism is an automorphism. For example, the infinite cyclic group is not co-Hopfian whereas all finite groups (in particular $T_2$) are Hopfian as well as co-Hopfian. For twin groups $T_n$ with $n \geq 3$, we have the following result.

\begin{theorem}\label{MainTheoremCoHopfian}
$T_n$ is not co-Hopfian for $n \geq 3$.
\end{theorem}
\begin{proof}
We construct a map $\psi_n : T_n \to T_n$, $n \geq 3$, by setting

\begin{equation*}
\psi_n(s_i) = 
\begin{cases}
s_i & \text{for}\ i\neq 2,\\
s_2s_1s_2 & \text{for}\ i=2.
\end{cases}
\end{equation*}

It is easy to check that $\psi_n$ is a group homomorphism. It is evident from the definition of $\psi_n$ that for any element $w\in T_n$, the expression $\psi_n(w)$ contains an even number of $s_2$. Hence, $s_2\notin \psi_n(T_n)$ and the map $\psi_n$ is not surjective. We now proceed to prove that $\psi_n$ is injective by induction on $n$.

We note that $T_3$ can be generated by elements $s_1$ and $s_1s_2$. Since the cyclic subgroup $\langle s_1s_2 \rangle$ is normal in $T_3$, any element of $T_3$ is of the form either $(s_1s_2)^m$ or ${(s_1s_2)^m}s_1$ for some integer $m$. The images of these elements under $\psi_3$ are:
\begin{align*}
\psi_3((s_1s_2)^m) &= (s_1s_2s_1s_2)^m = (s_1s_2)^{2m},&\\
\psi_3({(s_1s_2)^m}s_1)&= {(s_1s_2s_1s_2)^m}s_1 = (s_1s_2)^{2m}s_1,&
\end{align*}

It is clear that none of the non-trivial elements of $T_3$ belong to $\Ker(\psi_3)$, and hence $\psi_3$ is injective. \par

Suppose that $1 \neq w \in \Ker(\psi_4)$. Without loss of generality we may assume that $w$ is a reduced word.  Suppose that $w = w_1s_3w_2s_3\cdots w_ks_3w_{k+1}$, where $w_i$'s are words in $T_3$. Then, we have $$\psi_4(w)= \psi_4(w_1)s_3\psi_4(w_2)s_3\cdots \psi_4(w_k)s_3\psi_4(w_{k+1})=1.$$
Notice that all the $w_i$'s are non-trivial words in $T_3$, since $w$ is a reduced word. Also, the map $\psi_4$ restricted to $T_3$ is $\psi_3$ which is injective. Thus, $\psi_4(w_i) = \psi_3(w_i) \neq 1$ for all $1 \le i \le k+1$. For $\psi_4(w)=1$ to be true, all the $s_3$'s must get cancelled. But there will always be at least one $s_2$ in between any two $s_3$'s, which is a contradiction.  Therefore, the map $\psi_4$ is injective.
\par
Let us now assume that $\psi_{n-1}$ is injective for $n\ge 5$. Consider a non-trivial reduced word $w$ in $\Ker(\psi_n)$. Let $w=w_1s_{n-1}w_2s_{n-1}\cdots w_ks_{n-1}w_{k+1}$, where $w_i \in T_{n-1}$ for all $1 \le i \le k+1$. This implies that $$\psi_n(w)= \psi_{n-1}(w_1)s_{n-1}\psi_{n-1}(w_2)s_{n-1}\cdots \psi_{n-1}(w_k)s_{n-1}\psi_{n-1}(w_{k+1})=1.$$ For the above equality to be true, all the $s_{n-1}$'s must get cancelled. In particular, the two $s_{n-1}$'s in the subword $s_{n-1}\psi_{n-1}(w_j)s_{n-1}$ must get cancelled. This implies that $\psi_{n-1}(w_j)$ does not have $s_{n-2}$, which means that $w_j$ does not have $s_{n-2}$, which contradicts the fact that $w$ is reduced. Hence, $\psi_n$ is injective. 
\end{proof}

\begin{remark}
Note that the infinite cyclic group and a free product of any two non-trivial groups is not co-Hopfian. Thus, $PT_n$ is not co-Hopfian for $3\leq n \leq 6$. Whether $PT_n$ is co-Hopfian for $n \geq 7$ remains unknown.
\end{remark}

\begin{remark}
It is well-known that the braid group $B_n$ is not co-Hopfian for $n \geq 2$ \cite{bell}. In fact, the map $\phi_n: B_n \to B_n$, $n \geq 2$, defined on the standard generators by
$$\phi_n(\sigma_i)=\sigma_iz,$$ where $\langle z \rangle = \Z(B_n)$, is an injective homomorphism which is not surjective. Since $\phi_n(P_n) \subset P_n$ and $z \in \Z(B_n)=\Z(P_n)$ does not have a preimage under $\phi_n$, it follows that the restriction of $\phi_n$ on $P_n$ is injective but not surjective, and hence $P_n$ is not co-Hopfian for $n \geq 2$.
\end{remark}
\bigskip

\noindent\textbf{Acknowledgement.}
The authors thank the referee for many useful comments. Also, Timur Nasybullov is thanked for many comments and discussions on twisted conjugacy, in particular for Lemma \ref{alternate-condition}, which considerably shortened the original proof of Theorem \ref{MainTheoremRinfinity}. Tushar Kanta Naik
and Neha Nanda thank IISER Mohali for the Post Doctoral and the PhD Research Fellowships, respectively. Mahender Singh acknowledges support from the Swarna Jayanti Fellowship grants DST/SJF/MSA-02/2018-19 and SB/SJF/2019-20/04.

\end{document}